\documentclass[12pt,a4paper]{amsart}
\usepackage[latin1]{inputenc}
\usepackage{graphicx}
\usepackage{amssymb, amsmath}
\usepackage{geometry}
\usepackage{enumerate}
\newtheorem{theorem}{Theorem}[section]
\newtheorem{definition}[theorem]{Definition}
\newtheorem{lemma}[theorem]{Lemma}
\newtheorem{corollary}[theorem]{Corollary}
\newtheorem{proposition}[theorem]{Proposition}

\theoremstyle{definition}
\newtheorem{example}[theorem]{Example}
\newtheorem{remark}[theorem]{Remark}

\newcommand\remove[1]{}

\newcommand{\rnote}[1]{}
\newcommand{\jnote}[1]{}

\def\cprime{$'$}

\newcommand{\N}{\mathbb{N}}

\newcommand{\spn}{\mathrm{span}}
\newcommand{\sol}{\mathrm{sol}}

\newcommand{\dens}{\mathrm{dens}}

\numberwithin{equation}{section}

\begin{document}

\title[WCG Banach lattices]{Weakly compactly generated Banach lattices}

\author{A. Avil\'{e}s}
\address{Departamento de Matem\'{a}ticas\\ Facultad de Matem\'{a}ticas\\ Universidad de Murcia\\ 30100 Espinardo (Murcia)\\ Spain} 
\email{avileslo@um.es}

\author{A.J. Guirao}
\address{Instituto Universitario de Matem\'{a}tica Pura y Aplicada\\ Universitat Polit\`{e}cnica de Val\`{e}ncia\\ Camino de Vera s/n\\ 46022 Valencia\\ Spain}
\email{anguisa2@mat.upv.es}

\author{S. Lajara}
\address{Departamento de Matem\'{a}ticas\\Escuela de Ingenieros Industriales\\ Universidad de Castilla-La Mancha\\ 02071 Albacete\\ Spain} 
\email{sebastian.lajara@uclm.es}

\author{J. Rodr\'{i}guez}
\address{Departamento de Matem\'{a}tica Aplicada\\ Facultad de Inform\'{a}tica\\ Universidad de Murcia\\ 30100 Espinardo (Murcia)\\ Spain} 
\email{joserr@um.es}

\author{P. Tradacete}
\address{Departamento de Matem\'{a}ticas\\ Universidad Carlos III de Madrid\\ 28911 Legan\'es (Madrid)\\ Spain}
\email{ptradace@math.uc3m.es}

\thanks{Research partially supported by {\em Ministerio de Econom\'{i}a y Competitividad} and {\em FEDER} 
under projects MTM2014-54182-P (A.~Avil\'{e}s, S.~Lajara and J.~Rodr\'{i}guez), MTM2014-57838-C2-1-P (A.J.~Guirao), 
MTM2012-34341 (S.~Lajara), MTM2012-31286 (P.~Tradacete) and MTM2013-40985 (P.~Tradacete).
This work was also partially supported by the research projects 19275/PI/14 (A.~Avil\'{e}s, S.~Lajara and J.~Rodr\'{i}guez)
and 19368/PI/14 (A.J.~Guirao) funded by {\em Fundaci\'{o}n S\'{e}neca - Agencia de Ciencia y Tecnolog\'{i}a de la Regi\'{o}n de Murcia} 
within the framework of {\em PCTIRM 2011-2014}. The research of P.~Tradacete was also partially supported by Grupo UCM 910346}


\begin{abstract}
We study the different ways in which a weakly compact set can generate a Banach lattice. 
Among other things, it is shown that in an order continuous Banach lattice $X$, 
the existence of a weakly compact set $K\subset X$ such that $X$ coincides with the band generated by $K$, implies that $X$ is WCG.
\end{abstract}

\maketitle

\thispagestyle{empty}


\section{The general problem}

The purpose of this note is to study Banach lattices which are generated in one way or another by a weakly compact set. 
Namely, we will explore the connection between the existence of a weakly compact set which generates a Banach lattice as a linear space, 
a lattice, an ideal or a band. Our motivation starts with the question of J.~Diestel of whether 
every Banach lattice which is generated, as a lattice, by a weakly compact set must be weakly compactly generated (i.e., as a linear space).

Recall that a Banach lattice is a Banach space endowed with additional order and lattice structures 
which behave well with respect to the norm and linear structure. This is in particular highlighted by the 
fact that $\|x\|\leq\|y\|$ whenever $|x|\leq|y|$, or by the norm continuity of the lattice operations $\wedge$ and $\vee$. 
However, for the weak topology, the relation with the order and lattice structures is more subtle, 
in particular it is not always true that the lattice operations are weakly continuous. In fact, on infinite dimensional Banach lattices 
the weak topology fails to be locally solid (see e.g. \cite[Theorem~6.9]{ali-bur-2}).

A Banach space $X$ is called {\em weakly compactly generated} (WCG) whenever there exists a weakly compact subset of~$X$ 
whose closed linear span coincides with~$X$. This class of Banach spaces was first studied by Corson~\cite{cor-J} and 
it was pushed further by the fundamental work of Amir and Lindenstrauss~\cite{ami-lin}. 
Nowadays, WCG spaces play a relevant role in non-separable Banach space theory.
For complete information on WCG spaces, see \cite{fab-ultimo,fab-alt-JJ,ziz}.

Weakly compact sets and weakly compact operators in Banach lattices have been the object of research 
by several authors (cf. \cite{ali-bur:81,ali-bur:84,che-wic:00,nic}, 
see also the monographs \cite[Chapter 4.2]{ali-bur} and \cite[Chapter 2.5]{mey2}). 
In particular, WCG Banach lattices have been considered in \cite{buk-alt} and \cite{saa-saa:83}.

Before introducing the main notions of the paper let us recall that a {\em sublattice} 
of a Banach lattice $X$ is a subspace which is also closed under the lattice operations $\vee$ and $\wedge$. 
Also, an {\em ideal} $Y$ of $X$ is a subspace with the property that $|x|\leq|y|$ with $y\in Y$ implies that $x\in Y$. 
Finally, a {\em band} $Z$ of $X$ is an ideal for which $\sup(A)\in Z$ whenever $A\subset Z$ and $\sup(A)$ exists in~$X$. 
Unless otherwise mentioned, all subspaces, sublattices, ideals and bands in this paper are assumed to be closed. 
Given a subset $A$ of a Banach lattice~$X$, we will denote by $\overline{{\rm span}}(A)$, $L(A)$, $I(A)$ and $B(A)$ 
the smallest subspace (respectively, sublattice, ideal and band) of $X$ containing~$A$.

\begin{definition}\label{definition:WCGs}
Let $X$ be a Banach lattice. We will say that
\begin{enumerate}
  \item[(i)] $X$ is weakly compactly generated as a lattice (LWCG) if there is a weakly compact set $K\subset X$ such that $X=L(K)$.
  \item[(ii)] $X$ is weakly compactly generated as an ideal (IWCG) if there is a weakly compact set $K\subset X$ such that $X=I(K)$.
  \item[(iii)] $X$ is weakly compactly generated as a band (BWCG) if there is a weakly compact set $K\subset X$ such that $X=B(K)$.
\end{enumerate}
\end{definition}

Since for every set $A\subset X$ the inclusions $\overline{{\rm span}}(A)\subset L(A)\subset I(A)\subset B(A)$ hold, we clearly have 
$$
	WCG\Rightarrow LWCG\Rightarrow IWCG\Rightarrow BWCG.
$$
Our interest is whether the converse implications hold. The equivalence between LWCG and WCG for general Banach lattices
seems to be an open question which was raised by J.~Diestel during the conference ``Integration, Vector Measures 
and Related Topics IV'' held in La Manga del Mar Menor, Spain, 2011.

The paper is organized as follows: 

In Section~\ref{section:Basic} we provide a first approach
to the comparison between the notion of WCG Banach lattice and the weaker versions introduced above. 
For instance, we prove that LWCG=WCG for Banach lattices having weakly sequentially continuous lattice operations
(Theorem~\ref{theorem:wsc-operations}). We also show that, in general, IWCG$\neq$LWCG and BWCG$\neq$IWCG
(Examples~\ref{example:Eberlein} and~\ref{example:Lorentz}).

In Section~\ref{section:OC} we prove that BWCG=WCG for order continuous Banach lattices
(Theorem~\ref{t:order cont IWCG=WCG}). Some related results on Dedekind complete Banach lattices are also given. 
As a by-product of our methods we provide some applications to weakly precompactly generated Banach lattices.

In Section~\ref{section:DFJP} we apply the factorization method
of Davis-Figiel-Johnson-Pe\l czy\'{n}ski in our framework. For instance, it is shown that an IWCG Banach lattice
not containing $C[0,1]$ is Asplund generated (Theorem~\ref{theorem:AsplundGenerated}).

In Section~\ref{section:Miscellaneous} we collect some results
about the stability of weakly compact generation properties in Banach lattices. 
In general, the property of being LWCG is not inherited by sublattices. 
We discuss the three-space problem for LWCG Banach lattices (Example~\ref{example:3SP} and Theorem~\ref{theorem:3space})
and the connection of these properties with weakly Lindel\"{o}f determined Banach spaces.

We use standard Banach space/lattice terminology as can be found in \cite{ali-bur}, \cite{lin-tza-2} and~\cite{mey2}. 
By an {\em operator} between Banach spaces we mean a linear continuous map. 
The closed unit ball of a Banach space~$X$ is denoted by~$B_X$ and the dual of~$X$ is denoted by~$X^*$.
The weak$^*$ topology of~$X^*$ is denoted by~$w^*$.
The symbol $X_+$ stands for the positive cone of a Banach lattice $X$ and we write $C_+=C\cap X_+$
for any $C \subset X$.

\section{Basic approach}\label{section:Basic}

Given a Banach lattice $X$, for a set $A\subset X$ we define
\begin{displaymath}
	A^\wedge:=\Big\{\bigwedge_{i=1}^n a_i:\,n\in\mathbb N,\,(a_i)_{i=1}^n\subset A\Big\},
\end{displaymath}
\begin{displaymath}
	A^\vee:=\Big\{\bigvee_{i=1}^n a_i:\,n\in\mathbb N,\,(a_i)_{i=1}^n\subset A\Big\}.
\end{displaymath}
We will denote $A^{\wedge\vee}:=(A^\wedge)^\vee$ and $A^{\vee\wedge}:=(A^\vee)^\wedge$. Using the distributive law of the lattice operations, 
it is easy to see that $A^{\vee\wedge}=A^{\wedge\vee}$ and that 
\begin{equation}\label{eqn:L}
	L(A)=\overline{\spn(A)^{\vee\wedge}}
\end{equation}
(see e.g. \cite[p.~204]{ali-bur}). The {\em solid hull} $\sol(A)$ of~$A$ is the smallest solid subset of~$X$ containing~$A$, which
can be written as 
$$
	\sol(A)=\bigcup_{x\in A}[-|x|,|x|].
$$ 
It is not difficult to check that 
\begin{equation}\label{eqn:I}
	I(A)=\overline{\spn}(\sol(A)).
\end{equation}
The disjoint complement of $A$ is defined as  
$$
	A^\perp=\{x\in X: \, |x|\wedge|y|=0 \textrm{ for every }y\in A\}.
$$
It is well known that  
\begin{equation}\label{eqn:B}
	B(A)=A^{\perp\perp}
\end{equation}
(see e.g. \cite[Proposition 1.2.7]{mey2}).

Recall that an operator between Banach lattices $T:X\rightarrow Y$ is said to be: 
\begin{itemize}
\item {\em lattice homomorphism}, if $T(x_1\vee x_2)=(Tx_1)\vee (Tx_2)$ for every $x_1,x_2\in X$;
\item {\em interval preserving}, if it is positive and $T[0,x]=[0,Tx]$ for every $x\in X_+$.
\end{itemize}

\begin{proposition}\label{pro:onto}
Let $X$ and $Y$ be Banach lattices and $T:X\rightarrow Y$ an  operator with dense range. 
\begin{enumerate}
\item[(i)] If $X$ is LWCG and $T$ is a lattice homomorphism, then $Y$ is LWCG. 
\item[(ii)] If $X$ is IWCG and $T$ is an interval preserving lattice homomorphism, then $Y$ is IWCG.  
\end{enumerate}
\end{proposition}

\begin{proof}
(i) Since $T$ is a lattice homomorphism, we have
$L(T(A))=\overline{T(L(A))}$ for any $A \subset X$.
In particular, if $K\subset X$ is a weakly compact set such that $X=L(K)$, 
then $T(K)$ is a weakly compact set in~$Y$ such that $Y=\overline{T(X)}=L(T(K))$. 

(ii) Since $T$ is an interval preserving lattice homomorphism, 
$I(T(A))=\overline{T(I(A))}$ for any $A \subset X$.
Therefore, if $K\subset X$ is a weakly compact set such that $X=I(K)$, then
$T(K)$ is a weakly compact set in~$Y$ satisfying $Y=\overline{T(X)}=I(T(K))$.
\end{proof}

Recall that a Banach lattice is said to have {\em weakly sequentially continuous lattice operations} 
if $x_n\vee y_n$ converges weakly to $x\vee y$ whenever $(x_n)$ and $(y_n)$ converge weakly to $x$ and $y$, respectively. 
The basic examples of Banach lattices having weakly sequentially continuous lattice operations are AM-spaces 
(e.g. $C(K)$ spaces where $K$ is a compact Hausdorff topological space), see e.g. \cite[Theorem~4.31]{ali-bur}, 
and atomic order continuous Banach lattices (e.g. Banach spaces with unconditional basis), see e.g. \cite[Proposition~2.5.23]{mey2}.

\begin{theorem}\label{theorem:wsc-operations}
Let $X$ be a Banach lattice having weakly sequentially continuous lattice operations. Then $X$ is LWCG
if and only if it is WCG.
\end{theorem}

\begin{proof}
Let $K \subset X$ be a weakly compact set such that $L(K)=X$. By the Krein-Smulyan theorem
(see e.g. \cite[Theorem~3.42]{ali-bur}), we can assume that $K$ is absolutely convex. Hence
$\spn(K)=\bigcup_{n\in \N}nK$ is weakly $\sigma$-compact (that is, a countable union of weakly compact sets). 
Since $X$ has weakly sequentially continuous lattice operations,
for any weakly $\sigma$-compact set $A \subset X$ we have that both $A^\vee$ and $A^\wedge$ are weakly $\sigma$-compact.
In particular, $\spn(K)^{\vee\wedge}$ is weakly $\sigma$-compact
and since
$$
	X=L(K)\stackrel{\eqref{eqn:L}}{=}\overline{\spn(K)^{\vee\wedge}},
$$ 
we have that $X$ is WCG.
\end{proof}

\begin{corollary}\label{cor:CK}
Let $K$ be a compact Hausdorff topological space. Then:
\begin{enumerate}
\item[(i)] $C(K)$ is IWCG.
\item[(ii)] $C(K)$ is LWCG if and only if it is WCG.
\end{enumerate}
\end{corollary}
\begin{proof}
(i) follows from the fact that for the constant function $1_K$ we clearly have
$$
C(K)=I(\{1_K\}).
$$ 

(ii) is a direct consequence of Theorem~\ref{theorem:wsc-operations} and the comments preceding it.
\end{proof}

\begin{example}\label{example:Eberlein}
It is well known that $C(K)$ is WCG if and only if $K$ is Eberlein compact \cite{ami-lin} (cf. \cite[Theorem 14.9]{fab-ultimo}). 
If $\omega_1$ denotes the first uncountable ordinal, then the ordinal segment $[0,\omega_1]$ with its usual topology 
is a compact space which is not Eberlein. Thus, $C[0,\omega_1]$ provides an example of an IWCG Banach lattice which is not LWCG. 
Another example of this situation is given by the space $\ell_\infty$ (see also Corollary~\ref{pro:separablepredual} below).
\end{example}

In general, it is not true that the solid hull of a weakly relatively compact set is also weakly relatively compact
(see e.g. \cite[p.~108]{mey2}). Banach lattices with this stability property are 
order continuous and were characterized in~\cite[Theorem~2.4]{che-wic}: these include atomic order continuous 
Banach lattices, as well as Banach lattices not containing~$c_0$. 

\begin{theorem}\label{pro:solidhull}
Let $X$ be a Banach lattice with the property that the solid hull of any weakly relatively compact set is weakly relatively compact.
Then $X$ is BWCG if and only if it is WCG.
\end{theorem}
\begin{proof}
Since $X$ is order continuous, every ideal of~$X$ is a band
(see e.g. \cite[Corollary~2.4.4]{mey2}) and so $X$ is BWCG if and only if it is IWCG.
Let $K \subset X$ be a weakly compact set such that $X=I(K)$. Then $\sol(K)$ is weakly relatively compact and
$$
	X=I(K)\stackrel{\eqref{eqn:I}}{=}\overline{\spn}(\sol(K)),
$$ 
hence $X$ is WCG.
\end{proof}

It is clear that the discussion of this paper is only meaningful for non-separable Banach lattices. 
However, for Banach lattices with a separable predual we have some reformulations of the lattice versions of WCG, see
Corollary~\ref{pro:separablepredual} below. Recall first that a positive element $u$ 
of a Banach lattice~$X$ is said to be:
\begin{itemize}
\item a {\em quasi-interior point} of~$X$ if for every $x\in X_+$ we have $\|x-x\wedge nu\|\rightarrow 0$ as $n\to \infty$ 
or, equivalently, if $I(\{u\})=X$
(cf. \cite[Theorem~4.85]{ali-bur}); 
\item a {\em weak order unit} of~$X$ if $\{u\}^\perp=\{0\}$ or, equivalently, if $B(\{u\})=X$.
\end{itemize}  
In particular, every Banach lattice having a quasi-interior point (resp. weak order unit) is IWCG (resp. BWCG).

\begin{proposition}\label{proposition:GeneratedBySeparable}
Let $X$ be a Banach lattice. Then $X$ has a quasi-interior point (resp. weak order unit)
if and only if $X=I(C)$ (resp. $X=B(C)$) for some separable set $C \subset X$.
\end{proposition}

\begin{proof}
It suffices to prove the ``if'' parts. We can assume that $C$ is norm bounded.
Let $(x_n)_{n\in\mathbb N}$ be a dense sequence in~$C$ and define
$$
	u:=\sum_{n\in\mathbb N}\frac{|x_n|}{2^n} \in X_+.
$$
Since $x_n\in I(\{u\}) \subset B(\{u\})$ for all $n\in \N$, we have 
$$
	I(C) \subset I(\{u\})
	\quad\mbox{and} \quad
	B(C) \subset B(\{u\}).
$$
So, $u$ is a quasi-interior point (resp. weak order unit) of~$X$ whenever $X=I(C)$ (resp. $X=B(C)$).
\end{proof}

The {\em density character} of a topological space~$T$, denoted by~${\rm dens}(T)$,
is the minimal cardinality of a dense subset of~$T$. For an arbitrary Banach space~$X$ we have $\dens(X) \geq \dens(X^\ast,w^\ast)$
(see e.g. \cite[p.~576]{fab-ultimo}), while the equality 
$$
	\dens(X)=\dens(X^\ast,w^\ast)
$$
holds whenever $X$ is WCG (see e.g. \cite[Theorem~13.3]{fab-ultimo}). We next show that this equality holds
for any LWCG Banach lattice.

\begin{theorem}\label{dens LWCG}
Let $X$ be an LWCG Banach lattice. Then $\dens(X) = \dens(X^\ast,w^\ast)$.
\end{theorem}

\begin{proof}
It suffices to prove that $\dens(X) \leq \dens(X^\ast,w^\ast)$.
Let $K \subset X$ be a weakly compact set such that 
$$
	X=L(K)\stackrel{\eqref{eqn:L}}{=}\overline{\spn(K)^{\vee\wedge}}.
$$
Let us consider the WCG subspace $Y:=\overline{\spn}(K) \subset X$. 
According to the comments preceding the theorem, $\dens(Y) = \dens(Y^\ast,w^\ast)$.
Since $Y^{\vee\wedge}$ is dense in~$X$, we have
$\dens(Y)=\dens(X)$. Moreover, since the restriction operator $X^\ast\rightarrow Y^\ast$ is $w^\ast$-$w^\ast$-continuous 
and onto, we have $\dens(Y^\ast,w^\ast)\leq \dens(X^\ast,w^\ast)$.
It follows that $\dens(X) \leq \dens(X^\ast,w^\ast)$, as required.
\end{proof}

\begin{corollary}\label{pro:separablepredual}
Let $X$ be a Banach lattice such that $X^*$ is $w^*$-separable (e.g. $X=Y^*$ for a separable Banach lattice~$Y$). 
\begin{enumerate}
\item[(i)] $X$ is LWCG if and only if $X$ is separable.
\item[(ii)] $X$ is IWCG if and only if $X$ has a quasi-interior point.
\item[(iii)] $X$ is BWCG if and only if $X$ has a weak order unit.
\end{enumerate}
\end{corollary}
\begin{proof}
(i) is an immediate consequence of Theorem~\ref{dens LWCG}.
Since any weakly compact subset of~$X$ is separable, (ii) and (iii) follow from Proposition~\ref{proposition:GeneratedBySeparable}.
\end{proof}

The following illustrates the difference between BWCG and IWCG.

\begin{example}\label{example:Lorentz}
For $1<p<\infty$ the Lorentz space $L_{p,\infty}[0,1]$ is BWCG but not IWCG.
\end{example}

\begin{proof}
Recall that, for $1<p<\infty$, the Lorentz space $L_{p,\infty}[0,1]$ consists of those 
(equivalence classes of) measurable functions $f:[0,1]\rightarrow\mathbb R$ for which
$$
	\|f\|_{p,\infty}:=\sup_{t>0}t\lambda(\{x\in[0,1]: \, |f(x)|>t\})^{1/p}<\infty,
$$
where $\lambda$ denotes the Lebesgue measure on $[0,1]$.
Although the expression $\|f\|_{p,\infty}$ just defines a lattice quasi-norm, 
it is actually equivalent to a lattice norm (cf. \cite[p. 219, Lemma 4.5 and Theorem 4.6]{ben-sha}).

It is clear that $L_{p,\infty}[0,1]$ is BWCG since $\chi_{[0,1]}$ is a weak order unit of it. 
On the other hand, it is well known that $L_{p,\infty}[0,1]$ is the dual of a separable Banach lattice, namely,
the Lorentz space $L_{p',1}[0,1]$ with $\frac1p+\frac{1}{p'}=1$ (cf. \cite[p.~220, Theorem 4.7]{ben-sha}). 
Therefore, in order to prove that $L_{p,\infty}[0,1]$ is not IWCG it suffices to check that it has no
quasi-interior point (Corollary~\ref{pro:separablepredual}). Although this is probably known to any expert in the field,
we include a proof since we did not find a suitable reference for it.

Our proof is by contradiction. Suppose $L_{p,\infty}[0,1]$ has a quasi-interior point, say~$v$. 
Let us consider $f_0\in L_{p,\infty}[0,1]$ defined by $f_0(x):=\frac{1}{x^{1/p}}$ for $x\in[0,1]$. 
Observe that $\lambda(\{x\in[0,1]:f_0(x)>t\})=1/t^p$ for every $t>0$ and so $\|f_0\|_{p,\infty}=1$.
Set 
$$
	u:=\frac{v+f_0}{\|v+f_0\|_{p,\infty}}\in L_{p,\infty}[0,1].
$$
Clearly, $u$ is a quasi-interior point of~$L_{p,\infty}[0,1]$. Note that for any $t>0$ we have
$$
	\{x\in [0,1]: \, f_0(x) > t \|v+f_0\|_{p,\infty}\} \subset
	\{x\in [0,1]: \, u(x) > t \}
$$
and so, bearing in mind that $\|u\|_{p,\infty}=1$, we get 
$$
	\frac{1}{(t\|v+f_0\|_{p,\infty})^{p}} \leq 
	\lambda(\{x\in [0,1]: \, u(x) > t \}) \leq \frac{1}{t^{p}}.
$$
In view the previous inequalities, we can choose $t_0>0$ large enough such that
$$
	0<\lambda(\{x\in[0,1]: \, u(x)>t_0\})<1.
$$

Let $A_0:=\{x\in[0,1]:u(x)\leq t_0\}$, $A_1:=[0,1]\setminus A_0$ and $r_0:=\lambda(A_0)\in(0,1)$. 
There exists a measure-preserving transformation $\sigma:[0,1]\rightarrow[0,1]$ in such a way that $\sigma(A_0)=[0,r_0]$
and $\sigma(A_1)=[r_0,1]$ (see e.g. \cite[p.~81, Proposition~7.4]{ben-sha}).
Define $f_\sigma:=f\circ \sigma \in L_{p,\infty}[0,1]$. We claim that
\begin{equation}\label{eqn:notapprox}
	\|f_\sigma-f_\sigma\wedge Nu\|_{p,\infty}=1
\end{equation}
for every $N\in \mathbb N$. This would imply that $u$ cannot be a quasi-interior point, a contradiction.

In order to prove \eqref{eqn:notapprox}, note first that, since $\|f_\sigma\|_{p,\infty}=1$ and $u>0$, 
we have $\|f_\sigma-f_\sigma\wedge Nu\|_{p,\infty}\leq1$. For the converse inequality, fix $\varepsilon>0$ and choose $t>0$ large enough
such that 
$$
	\frac{1}{(t+Nt_0)^p} \leq r_0 \quad \mbox{and} \quad \frac{t}{t+Nt_0}\geq 1-\varepsilon.
$$
Define
$$
	B:=\sigma^{-1}\Big(\big[0,\frac{1}{(t+Nt_0)^p}\big)\Big)\subset A_0
$$
and note that for every $x\in B$ we have 
$$
	f_\sigma(x) >t+Nt_0 \geq t+Nu(x),
$$
hence $(f_\sigma \wedge Nu)(x)=Nu(x)$ and so $f_\sigma(x)-(f_\sigma \wedge Nu)(x)>t$.
It follows that
\begin{align*}
	\|f_\sigma-f_\sigma\wedge Nu\|_{p,\infty} 	&\geq t\lambda(\{x\in[0,1]: \, f_\sigma(x)-(f_\sigma\wedge Nu)(x)>t\})^{1/p} \\
	&\geq t\lambda(B)^{1/p}=\frac{t}{t+Nt_0}\geq 1-\varepsilon.
\end{align*}
As $\varepsilon>0$ is arbitrary, \eqref{eqn:notapprox} holds and the proof is complete.
\end{proof}

\section{Order continuous Banach lattices}\label{section:OC}

The next result provides an improvement of Theorem~\ref{pro:solidhull}.

\begin{theorem}\label{t:order cont IWCG=WCG}
Let $X$ be an order continuous Banach lattice. Then $X$ is BWCG if and only if it is WCG.
\end{theorem}

In order to prove Theorem~\ref{t:order cont IWCG=WCG} we need two lemmas. 
Recall that a subset $K$ of a Banach space is called {\em weakly precompact} 
(or conditionally weakly compact) if every sequence in~$K$ has a weakly Cauchy subsequence. 
Thanks to Rosenthal's $\ell_1$-theorem (see e.g. \cite[Theorem~5.37]{fab-ultimo}), this is equivalent
to saying that $K$ is bounded and contains no sequence equivalent to the usual basis of~$\ell_1$.

\begin{lemma}\label{lem:ocW}
Let $X$ be an order continuous Banach lattice, $K \subset X$ a weakly precompact set and 
$A\subset \sol(K)$ a set of pairwise disjoint vectors. Then $\sol(A)$ is weakly compact.
\end{lemma}
\begin{proof}
Let $(y_n)_{n\in\mathbb N}$ be a sequence in~$\sol(A) \subset \sol(K)$. By passing to a further subsequence, not relabeled, we can assume
that one of the following cases holds.

\emph{Case~1}. There is $x\in A$ such that $y_n\in [-|x|,|x|]$ for all $n\in \N$. Since
every order interval of an order continuous Banach lattice is weakly compact
(see e.g. \cite[Theorem~2.4.2]{mey2}), $(y_n)_{n\in\mathbb N}$ admits a subsequence which is weakly convergent to some vector in $[-|x|,|x|]\subset \sol(A)$.

\emph{Case~2}. There is a sequence $(x_n)_{n\in\mathbb N}$ of distinct elements of~$A$ in such a way that $y_n\in [-|x_n|,|x_n|]$ for all $n\in \N$.
In particular, $(y_n)_{n\in\mathbb N}$ is a disjoint sequence. Since $K$ is weakly precompact and $y_n\in \sol(K)$
for all $n\in \N$, the sequence $(y_n)_{n\in\mathbb N}$ is weakly convergent to $0\in \sol(A)$ (see e.g. \cite[Proposition~2.5.12(iii)]{mey2}).

This proves that $\sol(A)$ is weakly compact.
\end{proof}

\begin{lemma}\label{lem:ocSPAN}
Let $X$ be an order continuous Banach lattice, $C \subset X$ a solid set and 
$A \subset C_+$ a maximal set of pairwise disjoint vectors. Then $C \subset I(A)$.
\end{lemma}
\begin{proof}
We follow the ideas of \cite[Proposition 1.a.9]{lin-tza-2}. 
For each $x\in A$, let $P_x:X\to X$ be the band projection onto~$B(\{x\})$, so that
$$
	P_x(z)=\bigvee_{n\in\mathbb N}(z\wedge nx)=\lim_{n\to\infty} \bigvee_{k=1}^n(z\wedge kx)
$$
for all $z\in X_+$ (see e.g. \cite[pp. 8--10 and Proposition~1.a.8]{lin-tza-2}).

In order to see that $C \subset Y:=I(A)$ it is enough to prove that $C_+\subset Y$
(because $C$ is solid). To this end, pick $z\in C_+$. For every $x\in A$
we have $P_x(z)\in Y$ (bear in mind that $\bigvee_{k=1}^n(z\wedge kx)\in n\,\sol(A) \subset Y$ for all $n\in \N$)
and $0\leq P_x(z)\leq z$. Moreover, the sum $\sum_{x\in A} P_x(z)$ is unconditionally convergent to some $y \in [0,z]$ 
(see the proof of \cite[Proposition 1.a.9]{lin-tza-2}). We claim that $z=y$. Indeed, if this were not the case, 
then $z-y>0$ and, since $z-y\in C_+$ (bear in mind that $C$ is solid), by the maximality of $A$ there would be at least one $x\in A$ such that 
$x\wedge(z-y)\neq 0$. However, this is impossible since
$$
	0\leq x\wedge(z-y)\leq x\wedge(z-P_x(z))=0.
$$
Here, the last equality follows from the fact that $P_x$ is the band projection onto the band generated by~$x$. 
Hence, we have $z=y\in Y$. 
\end{proof}

\begin{proof}[Proof of Theorem~\ref{t:order cont IWCG=WCG}]
Suppose $X$ is BWCG. Since $X$ is order continuous, every ideal of~$X$ is a band
(see e.g. \cite[Corollary~2.4.4]{mey2}) and so $X$ is IWCG. 
Hence there is a weakly compact set $K\subset X$ such that 
$$
	X=I(K)\stackrel{\eqref{eqn:I}}{=}\overline{\spn}(\sol(K)).
$$ 
Fix a maximal set $A \subset \sol(K)_+$ of pairwise disjoint vectors.
By Lemma~\ref{lem:ocSPAN} (applied to $C:=\sol(K)$), we have $\sol(K) \subset I(A)=\overline{\spn}(\sol(A))$
and so $X=\overline{\spn}(\sol(A))$. Since $\sol(A)$ is weakly compact (by Lemma~\ref{lem:ocW}),
it follows that $X$ is WCG.
\end{proof}

\begin{remark}\label{remark:BWPG}
The proof of Theorem~\ref{t:order cont IWCG=WCG} makes clear that an order continuous Banach lattice~$X$ is WCG if and only if
there is a {\em weakly precompact} set $K \subset X$ such that $X=B(K)$. 
\end{remark}

Following~\cite[p.~28]{hay10}, a Banach space~$X$ is called {\em weakly precompactly generated} (WPG)
if there is a weakly precompact set $K \subset X$ such that $X=\overline{\spn}(K)$. 

\begin{corollary}\label{cor:ocWPG}
Let $X$ be an order continuous Banach lattice. Then $X$ is WCG if and only if it is WPG.
\end{corollary}

It is known that order continuous Banach lattices with order continuous dual are WCG, see~\cite[p.~194]{buk-alt}. 
We next provide another proof of this fact. For geometrical properties of this class
of Banach lattices, see \cite{gho}.

\begin{corollary}\label{cor:rusos}
Let $X$ be a Banach lattice. If $X$ and $X^*$ are order continuous, then $X$ is WCG.
\end{corollary}
\begin{proof}
The assumption implies that~$B_X$ is weakly precompact, see e.g. \cite[Theorem~4.25]{ali-bur}.
Hence $X$ is WPG and Corollary~\ref{cor:ocWPG} applies.
\end{proof}

Let us now turn the attention to the larger class of Dedekind complete (and $\sigma$-complete) Banach lattices.

\begin{theorem}\label{p:ell_infty}
Let $X$ be a Banach lattice and $Z \subset X$ a Dedekind complete sublattice. If $I(Z)$ is LWCG, then $Z$ is LWCG.
\end{theorem}
\begin{proof}
Note that
$$
	Y:=\left\{x\in X:\exists z\in Z \textrm{ with }|x|\leq z\right\}
$$
is the smallest (not necessarily closed) ideal of~$X$ containing~$Z$, so that $I(Z)=\overline{Y}$.
By the Lipecki-Luxemburg-Schep theorem (see e.g. \cite[Theorem 2.29]{ali-bur}), the identity on~$Z$ can be extended to a lattice homomorphism 
$T_0:Y\rightarrow Z$ (we use the Dedekind completeness of~$Z$ and the fact that $Z$ is a majorizing sublattice of~$Y$). By density,
$T_0$ admits a further extension to a lattice homomorphism $T:I(Z)\to Z$. Since $T$ is surjective
and $I(Z)$ is LWCG, Proposition \ref{pro:onto} ensures that $Z$ is LWCG.
\end{proof}

\begin{corollary}
Let $X$ be a Dedekind $\sigma$-complete Banach lattice. If every ideal of~$X$ is LWCG, then $X$ is WCG.
\end{corollary}

\begin{proof}
According to Theorem \ref{t:order cont IWCG=WCG}, it suffices to prove that $X$ is order continuous.
By contradiction, suppose $X$ is not order continuous. 
Since $X$ is Dedekind $\sigma$-complete, $X$ contains a sublattice $Z$ which is lattice isomorphic to~$\ell_\infty$
(see e.g. \cite[Theorem~4.51]{ali-bur}). In particular, $Z$ is Dedekind complete and non LWCG.
From Theorem~\ref{p:ell_infty} it follows that $I(Z)$ cannot be LWCG, a contradiction.
\end{proof}

These results motivate the question: {\em Can an LWCG Banach lattice contain a sublattice isomorphic to $\ell_\infty$?} 
If the answer were negative, then every Dedekind $\sigma$-complete 
LWCG Banach lattice would be WCG.

\section{Applications of the Davis-Figiel-Johnson-Pe\l czy\'{n}ski factorization}\label{section:DFJP}

The Davis-Figiel-Johnson-Pelcz\'{y}nski (DFJP)~\cite{dav-alt} factorization method is a keystone of Banach space theory.
Given an absolutely convex bounded subset $W$ of a Banach space~$X$, the DFJP interpolation Banach space
obtained from~$W$ is denoted by~$\Delta(W,X)$ (cf. \cite[Theorem~5.37]{ali-bur}). 
As a set, $\Delta(W,X)$ is a linear subspace of~$X$. The identity map $J:\Delta(W,X) \to X$ is an 
operator and $J(B_{\Delta(W,X)}) \supset W$. The space $\Delta(W,X)$ is reflexive
(resp. contains no isomorphic copy of~$\ell_1$) if and only if $W$ is weakly relatively
compact (resp. weakly precompact), see e.g. \cite[Theorem~5.37]{ali-bur} (resp. \cite[Theorem~5.3.6]{gon-abe}).

Bearing in mind that the absolutely convex hull of any weakly precompact set in a Banach space
is also weakly precompact (see e.g. \cite[p.~377]{ros-J-7}), it follows from the DFJP factorization method
that {\em a Banach space $X$ is WPG if and only if there exist a Banach space $Y$ not containing~$\ell_1$ 
and an operator $T:Y \to X$ with dense range}. As an application we get the following result
(cf. \cite[Corollary~2.3.1]{sch-PhD}).

\begin{proposition}\label{proposition:WPGNoBetaN}
If $X$ is a WPG Banach space, then $X$ contains no subspace isomorphic to~$\ell_\infty$.
\end{proposition}

\begin{proof} 
The property of being WPG is clearly inherited by {\em complemented} subspaces and, therefore, it suffices
to prove that $\ell_\infty$ is not WPG. By contradiction, suppose that $\ell_\infty$ is WPG.
Let $Y$ be a Banach space not containing~$\ell_1$ and $T:Y \to \ell_\infty$ an operator with dense range.
Then its adjoint $T^*: \ell_\infty^* \to Y^*$ is injective. In particular,
$(B_{Y^*},w^*)$ contains an homeomorphic copy of~$\beta\N$. Now, 
a result by Talagrand~\cite{tal4} ensures that $Y$ contains a subspace isomorphic to~$\ell_1(\mathfrak{c})$,
a contradiction.
\end{proof}

In order to apply Proposition~\ref{proposition:WPGNoBetaN} to Banach lattices, recall that
the following statements are equivalent for a Banach lattice~$X$ (see e.g. \cite[Theorem~4.69]{ali-bur}):
\begin{enumerate}
\item[(i)] $X^*$ is order continuous;
\item[(ii)] $X^*$ contains no subspace isomorphic to~$c_0$ (resp.~$\ell_\infty$);
\item[(iii)] $X$ contains no sublattice which is lattice isomorphic to~$\ell_1$. 
\end{enumerate}

\begin{corollary}\label{cor:dualWPG}
Let $X$ be a Banach lattice. Then $X^*$ is WCG if and only if it is WPG.
\end{corollary}
\begin{proof} In view of the comments above and Proposition~\ref{proposition:WPGNoBetaN}, if $X^*$ is WPG, then $X^*$ is order continuous and 
so the result follows from Corollary~\ref{cor:ocWPG}.
\end{proof}

The question of whether LWCG = WCG for arbitrary Banach lattices can be reduced to 
Banach lattices with order continuous dual, thanks to the following result.

\begin{theorem}\label{l:DFJP}
Let $X$ be a Banach lattice. If $X$ is LWCG (resp. IWCG), then there exist 
an LWCG (resp. IWCG) Banach lattice $Y$ and a lattice homomorphism (resp. interval preserving lattice homomorphism) $J:Y \to X$ such that:
\begin{enumerate}
\item[(i)] $Y^*$ is order continuous;
\item[(ii)] $X=\overline{J(Y)}$.
\end{enumerate}
\end{theorem}

\begin{proof}
Let $K \subset X$ be a weakly compact set such that $X=L(K)$ (resp. $X=I(K)$) and let $W:={\rm co}({\rm sol}(K))$ be its convex solid hull
(which is absolutely convex and bounded). Then $\Psi:=\Delta(W,X)$ is a Banach lattice, the identity operator $J: \Psi \to X$
is an interval preserving lattice homomorphism and $J(\Psi)$ is a
(not necessarily closed) ideal of~$X$, see e.g. \cite[Theorem~5.41]{ali-bur}. 

Moreover, from the weak compactness of~$K$ it follows that $\Psi^*$ is order continuous (see e.g. \cite[Theorem~5.43]{ali-bur}).
Since $J$ is a weak-weak homeomorphism when restricted to~$B_{\Psi}$ (see e.g. \cite[p.~313, Exercise~11]{ali-bur}),
the set $K_0=:J^{-1}(K)$ is weakly compact in~$\Psi$.
Then $Y:=L(K_0)$ (resp.~$Y:=I(K_0)$) is an LWCG sublattice (resp. IWCG ideal) of~$\Psi$. 
Since the property of having order continuous dual in inherited by sublattices
(see the comments preceding Corollary~\ref{cor:dualWPG}), $Y^*$ is order continuous. 
Finally, from the fact that $J$ is an interval preserving lattice homomorphism it follows that $X=\overline{J(Y)}$
(see the proof of Proposition~\ref{pro:onto}).
\end{proof}

\begin{remark}\label{remark:anotherproof}
The DFJP factorization and the result from~\cite{buk-alt} isolated in Corollary~\ref{cor:rusos}
provide an alternative proof of Theorem~\ref{t:order cont IWCG=WCG}. Indeed,
let $\Psi$ and $J$ be as in the proof of Theorem~\ref{l:DFJP}.
If we assume further that $X$ is order continuous, then so is $\Psi$ (see e.g. \cite[Theorem~5.41]{ali-bur}).
From the order continuity of $\Psi$ and $\Psi^*$ we infer that $\Psi$ is WCG, \cite[p.~194]{buk-alt}.
Finally, the equality $X=\overline{J(\Psi)}$ ensures that $X$ is WCG. 
\end{remark}

Order continuous Banach lattices cannot contain isomorphic copies of~$C[0,1]$ (see e.g. \cite[Corollary~5.1.12]{mey2}).
In Theorem~\ref{theorem:AsplundGenerated} below we give an improvement of Theorem~\ref{l:DFJP}
within the class of Banach lattices not containing~$C[0,1]$. 

Recall first that a Banach
space~$X$ is said to be {\em Asplund} if every separable subspace of~$X$ has separable dual or,
equivalently, $X^*$ has the Radon-Nikod\'{y}m property, \cite[p.~198]{die-uhl-J}. A Banach space~$X$
is said to be {\em Asplund generated} if there exist an Asplund Banach space~$Y$ and an operator $T:Y \to X$
with dense range. By the DFJP factorization, every WCG Banach space is Asplund generated.

\begin{theorem}\label{theorem:AsplundGenerated}
Let $X$ be a Banach lattice not containing subspaces isomorphic to~$C[0,1]$. If $X$ is LWCG (resp. IWCG), then there exist 
an LWCG (resp. IWCG) Banach lattice $Y$ and a lattice homomorphism (resp. interval preserving lattice homomorphism) $J:Y \to X$ such that:
\begin{enumerate}
\item[(i)] $Y$ is Asplund;
\item[(ii)] $X=\overline{J(Y)}$.
\end{enumerate}
In particular, $X$ is Asplund generated.
\end{theorem}

\begin{proof}
Fix a weakly compact set $K\subset X$ such that $X=L(K)$ (resp. $X=I(K)$) and consider the set $W:={\rm co}({\rm sol}(K))$. 
Since $X$ contains no isomorphic copy of~$C[0,1]$, the convex solid hull
of any weakly precompact subset of~$X$ is weakly precompact (see \cite[Corollary~II.4]{gho-joh}), and so
is~$W$. Let $\Psi$, $J$ and~$Y$ be as in the proof of Theorem~\ref{l:DFJP}. Since $W$ is weakly precompact, 
$\Psi$ contains no isomorphic copy of~$\ell_1$. Hence the Banach lattice $\Psi$ is Asplund (see \cite[p.~95]{die-uhl-J} and \cite[Theorem~7]{gho-saa-J})
and the same holds for its subspace~$Y$.
\end{proof}

In view of the previous theorem, if the equality LWCG = WCG were true for Asplund Banach lattices, then
it would also be true for all Banach lattices not containing isomorphic copies of~$C[0,1]$.

\section{Miscellaneous properties}\label{section:Miscellaneous}

Rosenthal~\cite{ros-J-8} gave the first instance of a WCG Banach space with a non WCG subspace. Likewise, LWCG/IWCG/BWCG are not hereditary properties: 

\begin{example}\label{example:AM}
Let $X$ be the Banach space constructed in \cite[Section 2]{arg-mer2}, which is WCG and has 
an uncountable unconditional basis $\mathcal{E}=\{e_{(\sigma,m)}:\sigma\in\N^\N,m\in\mathbb N\}$. 
In particular, $X$ is an LWCG Banach lattice. Define $x_\sigma:=\sum_{m\in \N}\frac{1}{2^{m/2}}\cdot e_{(\sigma,m)}$
for any $\sigma\in \N^\N$. In \cite[Theorem~2.6]{arg-mer2} it was proved that 
$\mathcal{B}=\{x_\sigma:\sigma\in\N^\N\}$ is a block basis of~$\mathcal{E}$ such that 
$Y:=\overline{\spn}(\mathcal{B})$ is not WCG. Note that $Y$ is a sublattice of~$X$ (because the coordinates of the $x_\sigma$'s 
with respect to~$\mathcal{E}$ are positive) which 
is not LWCG (by Theorem~\ref{theorem:wsc-operations}). In fact, $Y$ cannot be BWCG (by Theorem~\ref{t:order cont IWCG=WCG}).
\end{example}

It is well known that being WCG is not a three space property, that is, there exist non WCG Banach spaces~$X$
having a WCG subspace~$Y \subset X$ such that $X/Y$ is WCG.  For complete information on the three space
problem for WCG Banach spaces, see \cite[Section~4.10]{cas-gon} and the references therein. 
If $X$ is a Banach lattice and $Y \subset X$ is an {\em ideal}, then $X/Y$ is a Banach lattice
and the quotient operator $X \to X/Y$ is a lattice homomorphism (see e.g. \cite[p.~3]{lin-tza-2}). 
Some counterexamples to the three space problem for WCG spaces fit into the Banach lattice setting, like the following
construction which goes back to \cite{joh-lin} (cf. \cite[Section~4.10]{cas-gon}).

\begin{example}\label{example:3SP}
Let $2^{<\omega}$ be the dyadic tree (finite sequences of $0$s and $1$s), $2^\omega$ the set of its branches 
(countable infinite sequences of $0$s and $1$s) and $K$ the one-point compactification of
$2^{<\omega} \cup 2^\omega$ equipped with the topology defined by: (i)~all points from $2^{<\omega}$ are isolated;
(ii)~any $x=(x_k)_{k<\omega}\in 2^{\omega}$ has a neighborhood basis 
made of the sets $\{x\} \cup \{(x_k)_{k<m} : m>n\}$ for $n<\omega$. 
Then $L:=2^\omega \cup \{\infty\}$ is a closed subset of~$K$ and so $Y:=\{f\in C(K): f|_L\equiv 0\}$ is
an ideal of~$C(K)$. It is not difficult to check that $Y$ is isomorphic to $c_0$ and that 
the quotient space $C(K)/Y$ is isomorphic to $C(L)$ which, in turn, is isomorphic to~$c_0(\mathfrak{c})$.
Hence $Y$ and $C(K)/Y$ are WCG. On the other hand, $C(K)$ is not WCG, because it is not separable and 
every weakly compact subset of~$C(K)$ is separable (since $K$ is separable). By the same reason,
$C(K)$ is not LWCG (cf. Corollary~\ref{cor:CK}).
\end{example}

However, a Banach space $X$ is WCG if there is a {\em reflexive} subspace $Y \subset X$ such that $X/Y$ is WCG,
see \cite{joh-lin} (cf. \cite[Proposition 4.10.d]{cas-gon}). Theorem~\ref{theorem:3space} below
collects some positive results on the three space problem for WCG and LWCG Banach lattices. We first need a result
on WPG Banach spaces which might be of independent interest.

\begin{proposition}\label{proposition:3spaceWPG}
Let $X$ be a Banach space and $Y \subset X$ a subspace containing no isomorphic copy of~$\ell_1$. 
If $X/Y$ is WPG, then $X$ is WPG.
\end{proposition}

\begin{proof}
Let $q:X \to X/Y$ be the quotient operator and $K \subset X/Y$ a weakly precompact set such that $X/Y=\overline{{\rm span}}(K)$. 
Since $q$ is open and $K$ is bounded, there is a bounded set $G \subset X$ such that $q(G)=K$.
Since $Y$ contains no subspace isomorphic to~$\ell_1$ and $K$ is weakly precompact, $G$ is weakly precompact as well
(see e.g. \cite[2.4.a]{cas-gon}). Then $G_1:=G \cup B_Y \subset X$ is weakly precompact. 
We claim that $Z:=\overline{{\rm span}}(G_1)$ equals~$X$. By contradiction, suppose that $X\neq Z$.
By the Hahn-Banach separation theorem, there is $x^*\in X^* \setminus \{0\}$ such that $x^*(x)=0$ for all $x\in Z$.
In particular, $x^*$ vanishes on~$Y$ and so it factorizes as $x^*=\phi \circ q$ for some $\phi \in (X/Y)^*$. 
Note that $\phi$ vanishes on~$q(Z)$. But $X/Y=\overline{q(Z)}$ (because $q(Z)$ is a linear subspace
of~$X/Y$ containing $q(G)=K$), 
hence $\phi=0$ and so $x^*=0$, a contradiction. This shows that $X=Z$, as claimed. 
Therefore $X$ is WPG. 
\end{proof}

\begin{theorem}\label{theorem:3space}
Let $X$ be a Banach lattice and $Y \subset X$ an ideal.
\begin{enumerate}
\item[(i)] If $X$ is LWCG, then $X/Y$ is LWCG.
\item[(ii)] If $Y$ is reflexive and $X/Y$ is LWCG, then $X$ is LWCG.
\item[(iii)] If $X$ is order continuous, $Y$ contains no isomorphic copy of~$\ell_1$ and $X/Y$ is WCG, then $X$ is WCG.
\end{enumerate}
\end{theorem}

\begin{proof} (i) This follows at once from Proposition~\ref{pro:onto} because the quotient operator
$q:X \to X/Y$ is a surjective lattice homomorphism.

(ii) Let $K \subset X/Y$ be a weakly compact set such that $X/Y=L(K)$. 
Bearing in mind that $q$ is open and that $K$ is bounded and weakly closed, 
we can find a bounded and weakly closed set $K_0 \subset X$ such that $q(K_0)=K$.
Since $Y$ is reflexive and $K$ is weakly compact, $K_0$ is weakly compact as well
(see e.g. \cite[2.4.b]{cas-gon}). Then the set $K_1:=K_0 \cup B_Y \subset X$ is weakly compact. We claim that $X=L(K_1)$. Indeed, define 
$Z:=L(K_1)$. Since $q$ is a lattice homomorphism and $Z$ is a sublattice, $q(Z)$ is a (not necessarily closed)
sublattice of $X/Y$. Bearing in mind that $q(Z) \supset q(K_0) = K$, we conclude that $q(Z)$ is dense in~$X/Y$. 
As in the proof of Proposition~\ref{proposition:3spaceWPG}, it follows that $X=Z=L(K_1)$ and so $X$ is LWCG.

(iii) This follows from~Corollary~\ref{cor:ocWPG} and Proposition~\ref{proposition:3spaceWPG}. 
\end{proof}

In connection with part~(iii) of the previous theorem, note that
if $X$ is an order continuous Banach lattice and $Y \subset X$ is an ideal, then
the quotient space $X/Y$ is order continuous (see e.g. \cite[p.~205, Exercise~13]{ali-bur}).

A Banach space~$X$ is said to be {\em weakly Lindel\"{o}f determined} (WLD)
if $(B_{X^*},w^*)$ is a Corson compact, i.e. it is homeomorphic to a set~$S \subset [-1,1]^\Gamma$,
for some non-empty set~$\Gamma$, in such a way that $\{\gamma\in \Gamma:s(\gamma)\neq 0\}$ is countable for every $s\in S$.
Every WCG space is WLD, but the converse does not hold in general. For a complete account
on this class of Banach spaces, we refer the reader to~\cite{fab-J,fab-ultimo,fab-alt-JJ}.

\begin{theorem}\label{theorem:WLD}
Let $X$ be a Banach lattice such that the order intervals of~$X$ and~$X^*$ are separable and $w^*$-separable, respectively.
If there is a WLD subspace $Y \subset X$ such that $X=I(Y)$, then $X$ is WLD. 
\end{theorem}

Before proving Theorem~\ref{theorem:WLD}, let us mention that a Banach space 
is WCG if (and only if) it is Asplund generated and WLD, see e.g. \cite[Theorem~8.3.4]{fab-J}.
This fact together with Theorems~\ref{theorem:AsplundGenerated} and~\ref{theorem:WLD} yield the following:

\begin{corollary}\label{corollary:Rector}
Let $X$ be a Banach lattice such that the order intervals of~$X$ and~$X^*$ are separable and $w^*$-separable, respectively.
If $X$ is IWCG and $X$ contains no subspace isomorphic to~$C[0,1]$, then $X$ is WCG.
\end{corollary}

In order to prove Theorem~\ref{theorem:WLD} we first need two lemmas.
Given a Banach space~$X$, we say that a set $C \subset X$ {\em countably supports~$X^*$}
if for every $x^*\in X^*$ the set $\{x\in C: x^*(x)\neq 0\}$ is countable. 

\begin{lemma}\label{lemma:separable-intervals}
Let $X$ be a Banach lattice such that the order intervals of~$X^*$ are $w^*$-separable. 
If $C \subset X$ countably supports~$X^*$, then for every $x^*\in X^*$ the set $\{x\in C: x^*(|x|)\neq 0\}$ is countable. 
\end{lemma}

\begin{proof}
Since every element of~$X^*$ is the difference of two positive functionals, it suffices
to check that for every $x^*\in X^*_+$ the set $\{x\in C: x^*(|x|)\neq 0\}$ is countable. Fix
a $w^*$-dense sequence $(x_n^*)_{n\in \N}$ in~$[-x^*,x^*]$. Then for every $x\in X$ we have
$$
	x^*(|x|)=\sup\big\{y^*(x): \, y^*\in [-x^*,x^*]\big\}=
	\sup_{n\in \N} x_n^*(x)
$$
(see e.g. \cite[Theorem~1.23]{ali-bur}). Therefore
$$
	\{x\in C: \, x^*(|x|)\neq 0\}
	\subset \bigcup_{n\in \N} \{x\in C: \, x_n^*(x)\neq 0\},
$$
and so $\{x\in C: x^*(|x|)\neq 0\}$ is countable, as required.
\end{proof}

\begin{lemma}\label{lemma:separable-intervals2}
Let $X$ be a Banach lattice such that the order intervals of~$X$ and~$X^*$ are separable and $w^*$-separable, respectively.
If $C \subset X$ countably supports~$X^*$, then there is a set $P \subset {\rm sol}(C)$ such that 
${\rm sol}(C) \subset \overline{P}$ and $P$ countably supports~$X^*$.
\end{lemma}

\begin{proof}
For every $x\in C$ we take a countable dense set $A_x \subset [-|x|,|x|]$. Therefore, 
$P:=\bigcup_{x\in C}A_x$ is dense in ${\rm sol}(C)=\bigcup_{x\in C}[-|x|,|x|]$.
Fix $x^*\in X^*_+$. By Lemma~\ref{lemma:separable-intervals}, the set $C_0:=\{x\in C: x^*(|x|)\neq 0\}$
is countable. Since $x^*(y)=0$ for every $y\in [-|x|,|x|]$ whenever $x\in C \setminus C_0$, we have
$$
	\{y\in P: \, x^*(y)\neq 0\} \subset \bigcup_{x\in C_0}A_x,
$$
and so $\{y\in P: \, x^*(y)\neq 0\}$ is countable. As $x^*\in X^*_+$ is arbitrary and
every element of~$X^*$ is the difference of two positive functionals, $P$ countably supports~$X^*$.
\end{proof}

\begin{proof}[Proof of Theorem~\ref{theorem:WLD}]
Any WLD Banach space admits an M-basis, see e.g. \cite[Corollary~5.42]{fab-alt-JJ}.
Let $\{(y_i,y_i^*):i\in I\} \subset Y \times Y^*$ be an M-basis of~$Y$, that is, 
a biorthogonal system such that $Y=\overline{{\rm span}}(\{y_i:i\in I\})$ and $\{y_i^*:i\in I\}$ separates the points of~$Y$.
We can assume without loss of generality that $\|y_i\|\leq 1$ for all $i\in I$.
The fact that $Y$ is WLD ensures that $C:=\{y_i:i\in I\}$ countably supports~$X^*$ 
(see e.g. \cite[Theorem~5.37]{fab-alt-JJ}). Let $P \subset {\rm sol}(C)$ such that 
${\rm sol}(C) \subset \overline{P}$ and $P$ countably supports~$X^*$
(Lemma~\ref{lemma:separable-intervals2}). Since $X=I(Y)$, we have 
$$
	X=I(C)\stackrel{\eqref{eqn:I}}{=}\overline{{\rm span}}({\rm sol}(C))=\overline{{\rm span}}(P).
$$
It is now clear that the mapping
$$
	\phi: B_{X^*} \to [-1,1]^P, \quad
	\phi(x^*):=(x^*(x))_{x\in P},
$$
is a $w^*$-pointwise homeomorphic embedding witnessing that $(B_{X^*},w^*)$ is a Corson compact.
\end{proof}

Besides the separable case, the following Banach lattices have the property that the order intervals of their dual are $w^*$-separable:
\begin{enumerate}
\item[(i)] WLD Banach spaces with unconditional basis, like $c_0(\Gamma)$ and $\ell_p(\Gamma)$ for any $1<p<\infty$ and any non-empty set~$\Gamma$.
In this case, the order intervals of the dual have the stronger property of being $w^*$-metrizable.
\item[(ii)] $C(K)$, whenever $K$ is a compact space with the property that $L^1(\mu)$ is separable for 
every regular Borel probability $\mu$ on~$K$. This class of compact spaces includes all compacta which are Eberlein, Radon-Nikod\'{y}m, 
Rosenthal or linearly ordered, among others, see \cite{dza-kun,mar-ple:12,ple-sob}. In this case,
the order intervals of the dual are norm separable.
\end{enumerate}
On the other hand, it is not difficult to check that $L^1(\{0,1\}^{\omega_1})$ is a Banach lattice
for which the conclusion of Lemma~\ref{lemma:separable-intervals} fails.

\def\cprime{$'$}\def\cdprime{$''$}
  \def\polhk#1{\setbox0=\hbox{#1}{\ooalign{\hidewidth
  \lower1.5ex\hbox{`}\hidewidth\crcr\unhbox0}}} \def\cprime{$'$}
\providecommand{\bysame}{\leavevmode\hbox to3em{\hrulefill}\thinspace}
\providecommand{\MR}{\relax\ifhmode\unskip\space\fi MR }
\providecommand{\MRhref}[2]{%
  \href{http://www.ams.org/mathscinet-getitem?mr=#1}{#2}
}
\providecommand{\href}[2]{#2}

\bibliographystyle{amsplain}

\end{document}